\documentclass[reqno]{amsart}
\usepackage{amsmath, amssymb, amsthm, epsfig}
\usepackage{hyperref, latexsym}
\usepackage{url}
\usepackage[mathscr]{euscript}
\usepackage{enumerate}
\usepackage{color}
\usepackage{fullpage} 
\usepackage{setspace}
\usepackage{cancel}

\def\today{\ifcase\month\or
  January\or February\or March\or April\or May\or June\or
  July\or August\or September\or October\or November\or December\fi
  \space\number\day, \number\year}

\DeclareMathOperator{\supp}{\mathrm{supp}}

 \newtheorem{theorem}{Theorem}
  
 \newtheorem{lemma}[theorem]{Lemma}
 \newtheorem{proposition}[theorem]{Proposition}
 \newtheorem{corollary}[theorem]{Corollary}
 \theoremstyle{definition}

 \theoremstyle{remark}

 \newcommand{\R}{\mathbb{R}}
 \newcommand{\N}{\mathbb{N}}

 \newcommand{\hh}{\tfrac12}
 \newcommand{\ds}{\text{\rm d}s}

 \newcommand{\dt}{\text{\rm d}t}
  \renewcommand{\d}{\text{\rm d}}
 \newcommand{\du}{\text{\rm d}u}

 \newcommand{\dx}{\text{\rm d}x}
 
 \newcommand{\dy}{\text{\rm d}y}

\newcommand{\im}{{\rm Im}\,}
\newcommand{\re}{{\rm Re}\,}

\begin{document}
\title[Extreme values for $S_n(\sigma,t)$ near the critical line]{Extreme values for $S_n(\sigma,t)$ near the critical line}
\author[Chirre]{Andr\'{e}s Chirre}
\subjclass[2010]{11M06, 11M26, 11N37}
\keywords{Riemann zeta function, Riemann hypothesis, argument}
\address{IMPA - Instituto Nacional de Matem\'{a}tica Pura e Aplicada - Estrada Dona Castorina, 110, Rio de Janeiro, RJ, Brazil 22460-320}
\email{achirre@impa.br}

\allowdisplaybreaks
\numberwithin{equation}{section}

\maketitle

\begin{abstract}  Let $S(\sigma,t)=\frac{1}{\pi}\arg\zeta(\sigma+it)$ be the argument of the Riemann zeta function at the point  $\sigma+it$ of the critical strip. For $n\geq 1$ and $t>0$ we define
\begin{equation*}
S_{n}(\sigma,t) = \int_0^t S_{n-1}(\sigma,\tau) \,\d\tau\, + \delta_{n,\sigma\,},
\end{equation*}
where $\delta_{n,\sigma}$ is a specific constant depending on $\sigma$ and $n$. Let $0\leq \beta<1$ be a fixed real number. Assuming the Riemann hypothesis, we show lower bounds for the maximum of the function $S_n(\sigma,t)$ on the interval $T^\beta\leq t \leq T$ and near to the critical line, when $n\equiv 1\mod 4$. Similar estimates are obtained for $|S_n(\sigma,t)|$ when $n\not\equiv 1\mod 4$. This extends the results of Bondarenko and Seip \cite{BS} for a region near the critical line. In particular we obtain some omega results for these functions on the critical line. 
\end{abstract}

\section{Introduction}

In this paper, following similar ideas from Bondarenko and Seip \cite{BS}, we obtain new estimates for extreme values of the argument of the Riemann zeta function and its antiderivatives near the critical line assuming the Riemann hypothesis. Our main tools are convolution formulas for the functions $S_n(\sigma,t)$ and the version of the resonance method of Bondarenko and Seip given in \cite{BS}.

\smallskip

Let us begin by defining the main objects of our study and some results of them.

\subsection{Background} Let $\zeta(s)$ denote the Riemann zeta function. For $\hh\leq\sigma \leq1$ and $t>0$ we define
$$S(\sigma,t) = \tfrac{1}{\pi} \arg \zeta \big(\sigma + it \big),$$
where the argument is obtained by a continuous variation along straight line segments joining the points $2$, $2+i t$ and $\sigma + it$, assuming that this path has no zeros of $\zeta$, with the convention that $\arg \zeta(2) = 0$. If this path has zeros of $\zeta$ (including the endpoint $\sigma + it$) we set 
$$S(\sigma,t)  = \tfrac{1}{2}\, \lim_{\varepsilon \to 0} \left\{ S(\sigma,t + \varepsilon) + S(\sigma,t - \varepsilon)\right\}. $$
Useful information on the qualitative and quantitative behavior of $S(\sigma,t)$ is encoded in its antiderivatives. Setting $S_{0}(\sigma,t):=S(\sigma,t)$, for $n\geq 1$ we define, inductively, the functions
\begin{equation*}
S_{n}(\sigma,t) = \int_0^t S_{n-1}(\sigma,\tau) \,\d\tau\, + \delta_{n,\sigma\,},
\end{equation*}
where $\delta_{n,\sigma}$ is a specific constant depending on $\sigma$ and $n$. These are given by 
$$\delta_{2k-1,\sigma} =\frac{ (-1)^{k-1}}{\pi} \int_{\sigma}^{\infty} \int_{u_{2k-1}}^{\infty} \ldots \int_{u_{3}}^{\infty} \int_{u_{2}}^{\infty} \log |\zeta(u_1)|\, \du_1\,\du_2\,\ldots \,\du_{2k-1} $$
for $n = 2k-1$, with $k \geq 1$, and
$$\delta_{2k,\sigma} = (-1)^{k-1} \int_{\sigma}^{1} \int_{u_{2k}}^{1} \ldots \int_{u_{3}}^{1} \int_{u_{2}}^{1} \du_1\,\du_2\,\ldots \,\du_{2k} = \frac{(-1)^{k-1}(1-\sigma)^{2k}}{(2k)!}  $$ 
for $n = 2k$, with $k \geq 1$. 

\smallskip

Let $n\geq0$ be an integer and $\hh\leq\sigma\leq1$ be a fixed real number. We extend the functions $t\mapsto S_n(\sigma,t)$ to $\R$ in such a way that $S_n(\sigma,t)$ is an odd function when $n$ is even or is an even function when $n$ is odd. 

\smallskip
  
\subsection{Behavior on the critical line} When $\sigma=\hh$, we use the classical notation $S_n(\frac{1}{2},t)=S_n(t)$ and $S_0(t)=S(t)$. In 1924, J. E. Littlewood \cite[Theorem 11]{L} established, under the Riemann hypothesis (RH), the bound\footnote{The notation $f = O(g)$ (or $f \ll g$) means $|f(t)| \leq C \,g(t)$ for some constant $C>0$ and $t$ sufficiently large. In the subscript we indicate the parameters in which such constant $C$ may depend on.}
\begin{align} \label{7_5_12:10am}
S_n(t)=O_n\bigg(\dfrac{\log t}{(\log\log t)^{n+1}}\bigg),
\end{align} 
for $n\geq 0$. The order of magnitude of \eqref{7_5_12:10am} has never been improved and the efforts have been concentrated on optimizing the value of the implicit constants. The best known result for $n=0$ and $n=1$ is due to Carneiro, Chandee and Milinovich \cite{CCM1} and for $n\geq 2$ is due to Carneiro and Chirre \cite{CChi}. 

\medskip

On the other hand, for $n=0$ we have the following omega results\footnote{The notation $f =\Omega_{+}(g)$ means $f(t)> C \,g(t)$ for some constant $C>0$ and for some arbitrarily large values of $t$. The notation $f =\Omega_{-}(g)$ means $f(t)< -C \,g(t)$ for some constant $C>0$ and for some arbitrarily large values of $t$. The notation $f =\Omega_{\pm}(g)$ means that $f =\Omega_{+}(g)$ and $f =\Omega_{-}(g)$.}
\begin{align} \label{7_5_12:11am}
S(t)=\Omega_{\pm}\Bigg(\dfrac{(\log t)^{\frac{1}{2}}}{(\log\log t)^{\frac{1}{2}}}\Bigg),
\end{align}
established by Montgomery \cite[Theorem 2]{M}, under RH. It is likely that the estimate \eqref{7_5_12:11am} is closer to the behavior of the function $S(t)$ than the estimate \eqref{7_5_12:10am}. In fact, a heuristic argument by Farmer, Gonek and Hughes \cite{FGH} suggests that $S(t)$ grows as $(\log t \log\log t)^{\frac{1}{2}}$. Similarly, for the function $S_1(t)$ Tsang \cite[Theorem 5]{T} established, under RH, that
\begin{align*} 
S_1(t)=\Omega_{\pm}\Bigg(\dfrac{(\log t)^{\frac{1}{2}}}{(\log\log t)^{\frac{3}{2}}}\Bigg).
\end{align*}  
For the case $n\geq 2$, there are no known omega results for $S_n(t)$. 

\medskip

Recently, Bondarenko and Seip used their version of the resonance method with a certain convolution formula for $\zeta(s)$ to produce large values of the Riemann zeta function on the critical line \cite{BS}. Besides, using a convolution formula for $\log\zeta(s)$, they obtained similar results for the functions $S(t)$ and $S_1(t)$. They showed the following theorem.

\begin{theorem}[cf. Bondarenko and Seip \cite{BS}] \label{13_5_3:26am}
Assume the Riemann hypothesis. Let $0\leq\beta<1$ be a fixed real number. Then there exist two positive constants $c_0$ and $c_1$ such that, whenever $T$ is large enough,
$$
\max_{T^\beta\leq t\leq T}|S(t)| \geq c_0\dfrac{(\log T)^{\frac{1}{2}}(\log\log\log T)^{\frac{1}{2}}}{(\log\log T)^{\frac{1}{2}}}
$$
and
$$
\max_{T^\beta\leq t\leq T}S_1(t) \geq c_1\dfrac{(\log T)^{\frac{1}{2}}(\log\log\log T)^{\frac{1}{2}}}{(\log\log T)^{\frac{3}{2}}}.  
$$
\end{theorem}

\smallskip

Theorem \ref{13_5_3:26am} implies the following omega result\footnote{The notation $f =\Omega(g)$ means that $\lim_{t\to\infty} f(t)/g(t) \neq 0$.} for $S(t)$:
$$
S(t)=\Omega\Bigg(\dfrac{(\log t)^{\frac{1}{2}}(\log\log\log t)^{\frac{1}{2}}}{(\log\log t)^{\frac{1}{2}}}\Bigg).
$$
This result can be compared with the $\Omega_{\pm}$ results of Montgomery. For $S_1(t)$, Theorem \ref{13_5_3:26am} improved the $\Omega_{+}$ result given by Tsang by a factor $(\log\log\log t)^{\frac{1}{2}}$. 

\smallskip

\subsection{Behavior in the critical strip} Recently Carneiro, Chirre and Milinovich \cite[Theorem 2]{CChiM} showed new estimates for $S_n(\sigma,t)$ similar to \eqref{7_5_12:10am}. In particular, for a fixed number $\hh<\sigma<1$, under RH, we have that
$$
S_n(\sigma,t)=O_{n,\sigma}\bigg(\dfrac{{(\log t)}^{2-2\sigma}}{(\log\log t)^{n+1}}\bigg),
$$
for $n\geq 0$. On the other hand, under RH, Tsang \cite[Theorem 2 and p. 382]{T} states the following lower bound
\begin{align} \label{13_5_11:32pm} 
\sup_{t\in[T,2T]}\pm S(\sigma,t)\geq c\,\dfrac{(\log T)^{\frac{1}{2}}}{(\log\log T)^{\frac{1}{2}}},
\end{align}
for $\hh\leq \sigma \leq \hh + \frac{1}{\log \log T}$, $T$ sufficiently large and some constant $C>0$. This result shows extreme values for $S(\sigma,t)$ near the critical line. For the critical strip, a result of Montgomery \cite{M} states that, for a fixed $\hh<\sigma<1$, we have
$$
S(\sigma,t)=\Omega_{\pm}\bigg((\sigma-\hh)^2\dfrac{(\log t)^{1-\sigma}}{(\log \log t)^{\sigma}}\bigg).
$$

\medskip

 The main result of this paper is to show lower bounds for $S_n(\sigma,t)$ near the critical line, similar to \eqref{13_5_11:32pm}.
\begin{theorem} \label{23_4_8:21pm}
Assume the Riemann hypothesis. Let $0\leq \beta<1$ be a fixed number. Let $\sigma>0$ be a real number and $T>0$ sufficiently large in the range
\begin{align*} 
\dfrac{1}{2} \leq \sigma\leq \dfrac{1}{2}+\dfrac{1}{\log\log T}.
\end{align*}
Then there exists a sequence $\{c_n\}_{n\geq 0}$ of positive real numbers with the following property.
\begin{enumerate}
\item If $n\equiv 1\mod 4$:
$$
\max_{T^\beta\leq t\leq T}S_n(\sigma,t) \geq c_n\dfrac{(\log T)^{1-\sigma}(\log\log\log T)^{\sigma}}{(\log\log T)^{\sigma+n}}. 
$$
\item In the other cases:
$$
\max_{T^\beta\leq t\leq T}|S_n(\sigma,t)| \geq c_n\dfrac{(\log T)^{1-\sigma}(\log\log\log T)^{\sigma}}{(\log\log T)^{\sigma+n}}.  
$$
\end{enumerate} 
\end{theorem}
 
\medskip

Note that when $\sigma=\hh$ and $n=0$ or $1$, we recover Theorem \ref{13_5_3:26am}. Moreover, we obtain the new omega results on the critical line. 

\smallskip

\begin{corollary} Assume the Riemann hypothesis. Then 
\begin{enumerate}
\item If $n\equiv 1\mod 4$:
$$
S_n(t) = \Omega_{+}\Bigg(\dfrac{(\log t\log\log\log t)^{\frac{1}{2}}}{(\log\log t)^{n+\frac{1}{2}}}\Bigg).
$$
\item In the other cases:
$$
S_n(t) = \Omega\Bigg(\dfrac{(\log t\log\log\log t)^{\frac{1}{2}}}{(\log\log t)^{n+\frac{1}{2}}}\Bigg).
$$
\end{enumerate} 
\end{corollary}

\smallskip

\subsection{Strategy outline} Our approach is motivated by the ideas of Bondarenko and Seip \cite{BS} on the use of their version of the resonance method and a convolution formula for $\log\zeta(s)$. Soundararajan \cite{S} introduced the resonance method to produce large values of the Riemann zeta function on the critical line and large and small central values of $L$-functions. Also, this method has been the main tool for finding large values for the Riemann zeta function, $L$-functions and other objects related to them, in the critical strip (for instance in \cite{1, 2, 5, BS, BS2, BS3, 3, 4}). The main idea of the resonance method is to find a certain Dirichlet polynomial which ``resonates'' with the object to study. We will construct this Dirichlet polynomial in Section $4$. 

\smallskip

 The strategy can be broadly divided into the following three main steps:

\subsubsection{Step 1: Some results for $S_n(\sigma,t)$.} The first step is to show bounds for $S_n(\sigma,t)$ and for their moments. Bondarenko and Seip only needed to use the Littlewood's estimate \eqref{7_5_12:10am} and bounds of Selberg \cite{S1} for the moments of $S(t)$ and $S_1(t)$, assuming the Riemann hypothesis. In our case, we will use a weaker version of the result of Carneiro, Chirre and Milinovich \cite{CChiM}, to estimate the function $S_n(\sigma,t)$ uniformly in the critical strip. As a simple consequence of this result, we will obtain an estimate for its first moment. Finally, we will extend the convolution formula for $\log\zeta(s)$ given in \cite[Lemma 5]{T} for the function $S_n(\sigma,t)$. Although we restrict our attention to a region close to the critical line, we will show the bounds for $S_n(\sigma,t)$ in the critical strip, which may be of interest for other applications.
\subsubsection{Step 2: The resonator} The construction of our resonator is similar to that made by Bondarenko and Seip \cite[Section 3]{BS}. In particular, when $\sigma=\hh$ we obtain the resonator used by them. A deeper analysis in \cite[Lemmas 3 and 4]{BS} allows us to show these results for a region close to the critical line. This implies that the main relation between the resonator and the convolution formula of $S_n(\sigma,t)$ will follow immediately in the same way as obtained in the case $\sigma=\hh$ \cite[Lemma 7]{BS}.
\subsubsection{Step 3: Proof of Theorem \ref{23_4_8:21pm}} We follow the same outline in the proof of \cite[Theorem 2]{BS}. We will estimate the error terms in the integral that contains the resonator and the convolution formula of $S_n(\sigma,t)$. The main difference in our proof with that of Bondarenko and Seip is in the choice of the sign for a certain Gaussian kernel. This choice will depend on the remainder of $n$  modulo $4$. In particular, this allows to obtain $\Omega_{+}$ results for $S_n(t)$ when $n\equiv 1\mod 4$ and $\Omega$ results in the other cases.

\smallskip

Throughout this paper we will assume the Riemann hypothesis. Besides, for $f\in L^{1}(\R)$, we define the Fourier transform $\widehat{f}$ by
$$
\widehat{f}(\xi)=\int_{-\infty}^{\infty}f(x)e^{-2\pi i\xi x}\dx,
$$
for $\xi \in \R$.

\smallskip
 
\section{Some results for $S_n(\sigma,t)$} 
The main goal in this section is to show bounds for the functions $S_n(\sigma,t)$ and some convolution formulas of these functions with certain kernels. Throughout this section we let $n\geq0$ be an integer and $0<\delta\leq \hh$ be a real number.

\subsection{Bounds for $S_n(\sigma,t)$} The bounds that we will use for the functions $S_n(\sigma,t)$ will be a weaker version of a result of Carneiro, Chirre and Milinovich \cite{CChiM}. 

\begin{theorem} \label{19_4_4:21pm} Assume the Riemann hypothesis. We have the uniform bound
\begin{align*} 
S_n(\sigma,t)=O_{n,\delta}\bigg(\dfrac{(\log t)^{2-2\sigma}}{(\log\log t)^{n+1}}\bigg)
\end{align*}
in $\hh\leq \sigma \leq 1-\delta<1$ and $t>0$ sufficiently large. In particular, we obtain for all $t\in\R$ that
\begin{align}  \label{8_5_8:45pm}
S_n(\sigma,t) = O_{n,\delta}(\log(|t|+2)).
\end{align}
\end{theorem}
\begin{proof} It is enough to show when $\sigma>\frac{1}{2}$. For $t$ sufficiently large we have that
$$(1-\sigma)^2\log\log t\geq \delta^2\log\log t \geq 1.$$
Then, by \cite[Theorem 2]{CChiM} we have 
\begin{align} \label{7_5_4:43am}
\Big(-C_{n,\sigma}^{-}(t)+O_{n,\delta}(1)\Big)\dfrac{(\log\log t)^{2-2\sigma}}{(\log\log t)^{n+1}}\leq S_n(\sigma,t) \leq \Big(C_{n,\sigma}^{+}(t)+O_{n,\delta}(1)\Big)\dfrac{(\log\log t)^{2-2\sigma}}{(\log\log t)^{n+1}},
\end{align}
where $C_{n,\sigma}^{\pm}(t)$ are positive functions. For $n\geq 1$ odd, these functions are given by:
\begin{align} \label{7_5_4:37pm}
C^{\pm}_{n,\sigma}(t)=\dfrac{1}{2^{n+1}\pi}\bigg(H_{n+1}\Big(\pm(-1)^{\frac{n+1}{2}}(\log t)^{1-2\sigma}\Big) + \dfrac{2\sigma-1}{\sigma(1-\sigma)}\bigg),
\end{align}
where 
$$
H_n(x)=\displaystyle\sum_{k=0}^{\infty}\dfrac{x^k}{(k+1)^n}.
$$
Note that when $m\geq 2$, we have the bounds 
$$
1-\dfrac{1}{2^{m}}\leq H_m(x) \leq \zeta(m),
$$
for $|x|\leq1$. Therefore, we obtain in \eqref{7_5_4:37pm} for $n\geq 1$ odd and $t$ sufficiently large
\begin{align} \label{7_5_4:49am}
a_{n,\delta}\leq C_{n,\sigma}^{\pm}(t) \leq b_{n,\delta},
\end{align}
for some positive constants $a_{n,\delta}$ and $b_{n,\delta}$. Using \eqref{7_5_4:43am} we obtain the desired result in this case. For $n\geq 2$ even, these functions $C_{n,\sigma}^{\pm}(t)$ are given by:
$$
C_{n,\sigma}^{\pm}(t) = \left(\frac{2 \big(C_{n+1,\sigma}^{+}(t) + C_{n+1,\sigma}^{-}(t)\big) \, C_{n-1,\sigma}^{+}(t)\, C_{n-1,\sigma}^{-}(t)}{C_{n-1,\sigma}^{+}(t) + C_{n-1,\sigma}^{-}(t)}\right)^{\frac{1}{2}}.
$$
Since \eqref{7_5_4:49am} holds for $C_{n-1,\sigma}^{\pm}(t)$ and $C_{n+1,\sigma}^{\pm}(t)$, we have a similar estimate for $C_{n,\sigma}^{\pm}(t)$, and this implies the desired result in this case. When $n=0$ we have that 
\begin{equation*}
C_{0,\sigma}^{\pm}(t) =  \Big(2 \big(C_{1,\sigma}^{+}(t) + C_{1,\sigma}^{-}(t)\big) \, C_{-1,\sigma}(t)\Big)^{\frac{1}{2}},
\end{equation*}
where the function $C_{-1,\sigma}(t)$ is defined by
$$
C_{-1,\sigma}(t)=\dfrac{1}{\pi}\bigg(\dfrac{1}{1+(\log t)^{1-2\sigma}}+ \dfrac{2\sigma-1}{\sigma(1-\sigma)}\bigg).
$$
Using \eqref{7_5_4:49am} and a simple bound for $C_{-1,\sigma}(t)$, we bound $C_{0,\sigma}^{\pm}(t)$ and we conclude. Thefefore, it follows easily that \eqref{8_5_8:45pm} is valid for $t\geq t_0$ where $t_0$ is sufficiently large, and using the fact that the functions $S_n(\sigma,t)$ are bounded in $[\hh,1-\delta]\times [0,t_0]$ we conclude the proof.
\end{proof}

\smallskip

As a simple consequence we have the following estimate
\begin{align} \label{15_5_4:28pm}
\int_{0}^{T}|S_n(\sigma,t)|\dt = O_{n,\delta}(T\log T),
\end{align}
uniformly in $\hh\leq \sigma \leq 1-\delta<1$ and $T\geq 2$. Although this estimate is weak, it is sufficient for our purposes. For the case $\sigma=\hh$, better estimates are given by Littlewood \cite[Theorem 9 and p. 179]{L2} for all $n\geq 0$.

\smallskip

\subsection{Convolution formula} Now, we will obtain convolution formulas for the functions $S_n(\sigma,t)$ with certain kernels. The next lemma was introduced by Selberg \cite{S1}, and was also used by Tsang to study the functions $S(t)$ and $S_1(t)$ \cite{T,T3}. Since we assume the Riemann hypothesis, the factor that contains the zeros outside the critical line disappears.

\begin{lemma} \label{2:36_18_4} Assume the Riemann hypothesis. Suppose that $\hh\leq \sigma\leq 2$, and let $K(x+iy)$ be an analytic function in the horizontal strip $\sigma-2\leq y \leq 0$ satisfying the growth estimate
$$
V_\sigma(x):=\displaystyle\max_{\sigma-2\leq y \leq 0}|K(x+iy)|=O\bigg(\dfrac{1}{|x|\log^2|x|}\bigg)
$$
when $|x|\to\infty$. Then for every $t\neq 0$, we have 
\begin{align} \label{26_4_3:52pm}
\int_{-\infty}^{\infty}\log\zeta(\sigma+i(t+u))K(u)\du=\displaystyle\sum_{m= 2}^{\infty}\dfrac{\Lambda(m)}{m^{\sigma+it}\log m}\widehat{K}\bigg(\dfrac{\log m}{2\pi}\bigg)+ O\big(V_\sigma(-t)\big).
\end{align}
\end{lemma}
\begin{proof}
See \cite[Lemma 5]{T}.
\end{proof}

 
\smallskip

It is clear that the above lemma gives a convolution formula for the function $S(\sigma,t)$. To obtain a similar formula for the function $S_n(\sigma,t)$ when $n\geq 1$, we need an expression that connects the function $S_n(\sigma,t)$ with $\log\zeta(s)$. 

\begin{lemma} \label{2:30_18_4}
For $\hh\leq \sigma \leq 1$ and $t \neq 0$ we have
	\begin{equation*}
	S_n(\sigma,t) = \frac{1}{\pi} \,\,\im{\left\{\dfrac{i^{n}}{(n-1)!}\int_{\sigma}^{\infty}{\left(u-\sigma\right)^{n-1}\,\log\zeta(u+it)}\,\du\right\}}.
	\end{equation*}
\end{lemma}
\begin{proof} 
This follows from \cite[Lemma 6]{CChiM} and integration by parts.
\end{proof}

\smallskip

Using this expression we obtain the following convolution formula. This generalizes Tsang's conditional formula in \cite{T3} (or \cite[Eq. (10)]{BS}.

\begin{proposition} \label{20_4_3:44am} Assume the Riemann hypothesis and the same conditions for the function $K(x+iy)$ as in Lemma \ref{2:36_18_4}. Suppose further that $K$ is an even real-valued function (or odd real-valued function). Then for $\hh\leq \sigma\leq 1$ and $t\neq 0$, we have  
$$
\int_{-\infty}^{\infty}S_{n}(\sigma,t+s)\,K(s)\ds= \dfrac{1}{\pi}\im\bigg\{i^n\displaystyle\sum_{m= 2}^{\infty}\dfrac{\Lambda(m)}{m^{\sigma+it}(\log m)^{n+1}}\widehat{K}\bigg(\dfrac{\log m}{2\pi}\bigg)\bigg\}+O_n\big(V_{\frac{1}{2}}(t)+||K||_1\big).
$$
\end{proposition}
\begin{proof}
For the case $n=0$, we only need to take imaginary parts in \eqref{26_4_3:52pm}. For $n\geq1$, by Lemma \ref{2:30_18_4} we get 
\begin{align*} 
S_n(\sigma,t) & = \frac{1}{\pi} \,\,\im{\left\{\dfrac{i^{n}}{(n-1)!}\int_{\sigma}^{2}{\left(u-\sigma\right)^{n-1}\,\log\zeta(u+it)}\,\du\right\}} +  O_n(1).
\end{align*}
Plugging this in Lemma \ref{2:36_18_4} we obtain
\begin{align} \label{3:36_18_4}
\begin{split}
\int_{-\infty}^{\infty}S_{n}(\sigma,&t+s) \,K(s)\ds \\
& = \dfrac{1}{\pi}\int_{-\infty}^{\infty}\,\im{\left\{\dfrac{i^{n}}{(n-1)!}\int_{\sigma}^{2}{\left(u-\sigma\right)^{n-1}\,\log\zeta(u+i(t+s))}\,\du\right\}}K(s)\ds + O_n\big(||K||_1\big) \\
& = \dfrac{1}{\pi}\im{\left\{\dfrac{i^{n}}{(n-1)!}\int_{\sigma}^{2}\left(u-\sigma\right)^{n-1}\Bigg(\int_{-\infty}^{\infty}{\log\zeta(u+i(t+s))}K(s)\ds\Bigg)\du\right\}} + O_n\big(||K||_1\big)  \\
& = \dfrac{1}{\pi}\im{\left\{\dfrac{i^{n}}{(n-1)!}\int_{\sigma}^{2}\left(u-\sigma\right)^{n-1}\Bigg(\displaystyle\sum_{m=2}^{\infty}\dfrac{\Lambda(m)}{m^{u+it}\log m}\widehat{K}\bigg(\dfrac{\log m}{2\pi}\bigg)\Bigg)\du\right\}}   \\
 & \ \ \ \ + O_n\big(V_{\frac{1}{2}}(t)+||K||_1\big) \\
 & = \dfrac{1}{\pi}\im{\left\{\dfrac{i^{n}}{(n-1)!}\displaystyle\sum_{m=2}^{\infty}\dfrac{\Lambda(m)}{m^{it}\log m}\widehat{K}\bigg(\dfrac{\log m}{2\pi}\bigg)\bigg(\int_{\sigma}^{2}\dfrac{\left(u-\sigma\right)^{n-1}}{m^u}\du\bigg)\right\}}  \\
 & \ \ \ \ + O_n\big(V_{\frac{1}{2}}(t)+||K||_1\big),
\end{split}
\end{align}
where the interchange of the integrals is justified by Fubini's theorem, considering the estimates \cite[Theorem 13.18, Theorem 13.21]{MV}. Using  \cite[\S 2.321 Eq.2] {GR}) we obtain that 
\begin{align*} 
\int_{\sigma}^{2}{\dfrac{(u-\sigma)^{n-1}}{m^{u}}\,}\,\du = \dfrac{\beta_{n-1}}{m^\sigma (\log m)^n} - \dfrac{1}{m^{2}}\displaystyle\sum_{k=0}^{n-1}\dfrac{\beta_k}{(\log m)^{k+1}}(2-\sigma)^{n-1-k},
\end{align*}
where $\beta_k=\frac{(n-1)!}{(n-1-k)!}$. This implies that for each $m\geq 2$ we get
\begin{align*}
\int_{\sigma}^{2}\dfrac{(u-\sigma)^{n-1}}{m^{u}}\,\du & = \dfrac{(n-1)!}{m^{\sigma}(\log m)^{n}} + O_n\bigg(\dfrac{1}{m^{\frac{3}{2}}(\log m)^n}\bigg).
\end{align*}
Inserting this in \eqref{3:36_18_4}, and considering that $||\widehat{K}||_{\infty}\leq ||K||_1$, we obtain the desired result.
\end{proof}

\medskip

\section{The Resonator}
In this section we will construct the resonator. The construction of our resonator is similar to the resonator developed by Bondarenko and Seip \cite[Section 3]{BS}. The results presented here are extensions of their results, for a region near the critical line. The resonator is the function of the form $|R(t)|^2$, where
\begin{align} \label{7:40_16_4}
R(t)=\displaystyle\sum_{m\in\mathcal{M}'}r(m)m^{-it},
\end{align}
and $\mathcal{M}'$ is a suitable finite set of integers. Let $\sigma$ be a positive real number and $N$ be a positive integer sufficiently large, such that 
\begin{align} \label{1_5_9:20pm}
\dfrac{1}{2} \leq \sigma\leq \dfrac{1}{2}+\dfrac{1}{\log\log N}.
\end{align}
Our resonator will depend of $\sigma$ and $N$. For simplicity of notation, we write $\log_2x:=\log\log x$ and $\log_3x:=\log\log\log x$. Let $P$ be the set of prime numbers $p$ such that
\begin{align} \label{20_4_6:46pm}
e\log N\log_{2}N < p \leq \exp\big((\log_2N)^{1/8}\big)\log N\log_2N.
\end{align}
We define $f(n)$ to be the multiplicative function supported on the set of square-free numbers such that
$$
f(p):=\bigg({\dfrac{(\log N)^{1-\sigma}{(\log_2N)}^{\sigma}}{(\log_3N)^{1-\sigma}}}\bigg)\dfrac{1}{p^{\sigma}\,(\log p-\log_2N-\log_3N)},
$$
for $p\in P$ and $f(p)=0$ otherwise. For each $k\in\big\{1,\cdot\cdot\cdot,\big[(\log_2N)^{1/8}\big]\big\}$ we define the following sets:
\begin{align*}
P_k:=\big\{p: \mbox{prime number such that} \hspace{0.1cm} e^k\log N\log_2N<p\leq e^{k+1}\log N\log_2N\big\},
\end{align*}
\begin{align*}
M_k:=\bigg\{n\in\supp(f): n  \hspace{0.1cm} \mbox{has at least} \hspace{0.1cm} \alpha_{k}:=\frac{3(\log N)^{2-2\sigma}}{k^2(\log_3N)^{2-2\sigma}} \hspace{0.1cm} \mbox{prime divisors in}  \hspace{0.1cm} P_k\bigg\},
\end{align*}
\begin{align*}
M'_k:=\big\{n\in M_k: n  \hspace{0.1cm} \mbox{only has prime divisors in} \hspace{0.1cm} P_k\big\}.
\end{align*}
Finally, we define the set
\begin{align*}
\mathcal{M}:=\supp(f) \backslash \bigcup_{k=1}^{[(\log_2N)^{1/8}]}M_k .
\end{align*}
Note that if $m\in\mathcal{M}$ and $d|m$ then $d\in\mathcal{M}$.

\begin{lemma} \label{30_7_12:26pm} We have that
$$
|\mathcal{M}|\leq N,
$$
where $|\mathcal{M}|$ represents the cardinality of $\mathcal{M}$.
\end{lemma}
\begin{proof} The proof follows the same outline that \cite[Lemma 2]{BS2}. The main difference is the appearance of the term $(\log_3N)^{2\sigma-1},$ which is well estimated, whenever \eqref{1_5_9:20pm} holds. It allows us to obtain the same estimate for the cardinality of $\mathcal{M}$ as the case $\sigma=\hh$. By \cite[Eq. (9)-(10)]{BS2}, we have that
$$
{[x] \choose [y]} \leq \exp\big(y(\log x-\log y) + 2y + \log x\big),
$$
for $1\leq y\leq x$ and 
$$
2{m \choose n-1} \leq {m \choose n},
$$
for $3n-1\leq m$. By the prime number theorem, the cardinality of each $P_k$ is at most $e^{k+1}\log N$. Therefore, using the above inequalities and \eqref{1_5_9:20pm}
\begin{align*}
|\mathcal{M}| & \leq \displaystyle\prod_{k=1}^{[(\log_2N)^{1/8}]}\displaystyle\sum_{j=0}^{[\alpha_k]}{\big[e^{k+1}\log N\big] \choose j}\leq \displaystyle\prod_{k=1}^{[(\log_2N)^{1/8}]}2\,{\big[e^{k+1}\log N\big] \choose [\alpha_k]} \nonumber \\
& \leq \exp\Bigg(\displaystyle\sum_{k=1}^{[(\log_2N)^{1/8}]}\frac{3(\log N)^{2-2\sigma}}{(\log_3N)^{2-2\sigma}}\Bigg(\dfrac{1}{k}+\dfrac{3+2\log k}{k^2}+\dfrac{(2\sigma-1)\log_2N}{k^2}+ \dfrac{(2-2\sigma)\log_4N}{k^2}\Bigg)+3k+\log_2N\Bigg) \nonumber  \\
& \leq \exp\Bigg(\bigg(\dfrac{3}{4}+o(1)\bigg)(\log N)^{2-2\sigma}(\log_3N)^{2\sigma-1}\Bigg) \leq  \exp\Bigg(\bigg(\dfrac{3}{4}+o(1)\bigg)(\log N)(\log_3N)^{2/\log_2N}\Bigg).
\end{align*}
Then, for $N$ sufficiently large we get that $|\mathcal{M}|\leq N$.
\end{proof}

\smallskip

\begin{lemma} \label{2_5_2:20am}
For all $k=1,\cdot\cdot\cdot,[(\log_2N)^{1/8}]$ we have, as $N\to \infty$
$$
\displaystyle\sum_{p\in P_k}\dfrac{1}{p^{2\sigma}}=(1+o(1))\int_{e^k\log N\log_2N}^{e^{k+1}\log N\log_2N}\dfrac{1}{y^{2\sigma}\log y}\dy,
$$
where $o(1)$ is independent of $k$. In particular, we have that
\begin{align} \label{10_5_1:20am}
(d+o(1))\dfrac{1}{(\log_2N)^{2\sigma}}<\displaystyle\sum_{p\in P_k}\dfrac{1}{p^{2\sigma}}<(2+o(1))\dfrac{1}{(\log_2N)^{2\sigma}},
\end{align}
for some constant $0<d<1$.
\end{lemma}
\begin{proof}
Using \cite[Theorem 13.1]{MV}, under the Riemann hypothesis we have  
$$
\pi(x)=\displaystyle\int_{2}^{x}\dfrac{1}{\log y}\dy + O\big(x^\frac{1}{2}\log x\big),
$$
where $\pi(x)$ is the function that counts the prime numbers not exceeding $x$. Then, using integration by parts we get 
\begin{align*} 
\displaystyle\sum_{p\in P_k}\dfrac{1}{p^{2\sigma}}& =\int_{e^k\log N\log_2N}^{e^{k+1}\log N\log_2N}\dfrac{1}{y^{2\sigma}\log y}\dy +  O\Bigg(\int_{e^k\log N\log_2N}^{e^{k+1}\log N\log_2N}\dfrac{\log y}{y^{2\sigma+\frac{1}{2}}}\dy\Bigg) \nonumber \\
& = \bigg(1+O\bigg(\dfrac{1}{(\log N)^{1/4}}\bigg)\bigg)\int_{e^k\log N\log_2N}^{e^{k+1}\log N\log_2N}\dfrac{1}{y^{2\sigma}\log y}\dy.
\end{align*}
Now we can see that
$$
\int_{e^k\log N\log_2N}^{e^{k+1}\log N\log_2N}\dfrac{1}{y^{2\sigma}\log y}\dy\leq \dfrac{e^k\log N\log_2N(e-1)}{(e^k\log N\log_2N)^{2\sigma}\log\big(e^k\log N\log_2N\big)} < \dfrac{2}{(\log_2N)^{2\sigma}}.
$$
On the other hand, we know that $(e^{k}\log N)^{2\sigma-1}< (\log N)^{4\sigma-2}\leq e^{4}$ for all $1\leq k\leq [(\log_2N)^{1/8}]$. Therefore 
$$
\int_{e^k\log N\log_2N}^{e^{k+1}\log N\log_2N}\dfrac{1}{y^{2\sigma}\log y}\dy\geq \dfrac{e^k\log N\log_2N(e-1)}{(e^{k+1}\log N\log_2N)^{2\sigma}\log\big(e^{k+1}\log N\log_2N\big)} > \dfrac{d}{(\log_2N)^{2\sigma}},
$$
for some constant $0<d<1$. 
\end{proof}

\smallskip
The following lemma can be considered as an extension of \cite[Lemma 4]{BS} to the region \eqref{1_5_9:20pm}.
\begin{lemma} \label{10_5_1:57am}
We have
$$
\dfrac{1}{\displaystyle\sum_{l\in\N}f(l)^2}\displaystyle\sum_{n\in\mathcal{M}}f(n)^2\,\displaystyle\sum_{p |n}\dfrac{1}{f(p)\,p^{\sigma}}\geq c\,\dfrac{(\log N)^{1-\sigma}(\log_3N)^{\sigma}}{ (\log_2N)^{\sigma}},
$$
for some universal constant $c>0$.
\end{lemma}
\begin{proof} The proof is similar to \cite[Lemma 4]{BS}. For each $k\in\big\{1,\cdot\cdot\cdot,\big[(\log_2N)^{1/8}\big]\big\}$ we define the following sets:
\begin{align*}
L_k:=\bigg\{n\in\supp(f): n  \hspace{0.1cm} \mbox{has at most} \hspace{0.1cm} \beta_{k}:=\frac{d\,(\log N)^{2-2\sigma}}{12k^2(\log_3N)^{2-2\sigma}} \hspace{0.1cm} \mbox{prime divisors in}  \hspace{0.1cm} P_k\bigg\},
\end{align*}
where $d$ is the constant mentioned in Lemma \ref{2_5_2:20am}, and 
\begin{align*}
L'_k:=\big\{n\in L_k: n  \hspace{0.1cm} \mbox{only has prime divisors in} \hspace{0.1cm} P_k\big\}.
\end{align*}
Finaly, we define the set 
$$
\mathcal{L}:=\mathcal{M}\backslash \bigcup_{k=1}^{[(\log_2N)^{1/8}]}L_k.
$$
Now to prove the lemma, it is enough to show that
\begin{align} \label{30_4_5:9pm}
\dfrac{1}{\displaystyle\sum_{l\in\N}f(l)^2}\displaystyle\sum_{n\notin\mathcal{L}}f(n)^2= o(1),  \hspace{0.2cm} N \to \infty.
\end{align}
Indeed, using \eqref{30_4_5:9pm} and the fact that $\mathcal{L}\subset\mathcal{M}$ we get 
\begin{align*}
\dfrac{1}{\displaystyle\sum_{l\in\N}f(l)^2}\displaystyle\sum_{n\in\mathcal{M}}f(n)^2\,\displaystyle\sum_{p |n}\dfrac{1}{f(p)\,p^{\sigma}}& \geq \dfrac{1}{\displaystyle\sum_{l\in\N}f(l)^2}\displaystyle\sum_{n\in\mathcal{M}}f(n)^2\min_{n\in\mathcal{L}}\,\displaystyle\sum_{p |n}\dfrac{1}{f(p)\,p^{\sigma}} \nonumber \\
& \geq \big(1-o(1)\big)\min_{n\in\mathcal{L}}\,\displaystyle\sum_{p |n}\dfrac{1}{f(p)\,p^{\sigma}} \nonumber \\
& =  \big(1-o(1)\big)\displaystyle\sum_{k=1}^{[(\log_2N)^{1/8}]}\dfrac{d\,(\log N)^{2-2\sigma}}{12k^2(\log_3N)^{2-2\sigma}}\min_{p\in P_k}\dfrac{1}{f(p)\,p^{\sigma}}  \nonumber \\
& \geq \big(1-o(1)\big)\displaystyle\sum_{k=1}^{[(\log_2N)^{1/8}]}\dfrac{d\,(\log N)^{2-2\sigma}}{12k^2(\log_3N)^{2-2\sigma}}\bigg(\dfrac{k(\log_3N)^{1-\sigma}}{(\log N)^{1-\sigma}(\log_2N)^{\sigma}}\bigg)\nonumber \\
& \geq c\,\dfrac{(\log N)^{1-\sigma}(\log_3N)^{\sigma}}{ (\log_2N)^{\sigma}},
\end{align*}
for some constant $c>0$. Therefore, it remains to prove \eqref{30_4_5:9pm}. Since $$
\mathcal{L}:=\supp(f)\backslash \bigcup_{k=1}^{[(\log_2N)^{1/8}]}\big(M_k\cup L_k\big),
$$
it is enough to prove that when $N \to \infty$
\begin{align} \label{30_4_5:11pm}
\dfrac{1}{\displaystyle\sum_{l\in\N}f(l)^2}\displaystyle\sum_{k=1}^{[(\log_2N)^{1/8}]}\displaystyle\sum_{n\in L_k}f(n)^2= o(1),  
\end{align}
and
\begin{align} \label{30_4_5:12pm}
\dfrac{1}{\displaystyle\sum_{l\in\N}f(l)^2}\displaystyle\sum_{k=1}^{[(\log_2N)^{1/8}]}\displaystyle\sum_{n\in M_k}f(n)^2= o(1).
\end{align}
First we will prove \eqref{30_4_5:11pm}. For each $k\in\big\{1,\cdot\cdot\cdot,\big[(\log_2N)^{1/8}\big]\big\}$ and for any $0<b<1$ we have 
\begin{align} \label{1_5_3:36pm}
\begin{split}
\dfrac{1}{\displaystyle\sum_{l\in\N}f(l)^2}\displaystyle\sum_{n\in {L}_k}f(n)^2 & =\dfrac{1}{\displaystyle\prod_{p\in P_k}(1+f(p)^2)}\displaystyle\sum_{n\in L'_k}f(n)^2 \leq  b^{-\beta_k}\displaystyle\prod_{p\in P_{k}}\dfrac{\big(1+bf(p)^2\big)}{(1+f(p)^2)} \\
& \leq b^{-\beta_k}\exp\Bigg((b-1)\displaystyle\sum_{p\in P_{k}}\frac{f(p)^2}{1+f(p)^2}\Bigg).
\end{split}
\end{align}
Since $f(p)\leq 1$, using the left-hand side inequality of \eqref{10_5_1:20am} we get
\begin{align*}
\displaystyle\sum_{p\in P_{k}}\frac{f(p)^2}{1+f(p)^2}& \geq \dfrac{1}{2}\displaystyle\sum_{p\in P_{k}}f(p)^2=  \bigg({\dfrac{(\log N)^{2-2\sigma}{(\log_2N)}^{2\sigma}}{2(\log_3N)^{2-2\sigma}}}\bigg)\displaystyle\sum_{p\in P_{k}}\dfrac{1}{p^{2\sigma}\,(\log p-\log_2N-\log_3N)^2} \nonumber \\
& \geq \bigg({\dfrac{(\log N)^{2-2\sigma}}{8k^2(\log_3N)^{2-2\sigma}}}\bigg)(d+o(1)).\end{align*}
This implies in \eqref{1_5_3:36pm} that 
\begin{align*}
\dfrac{1}{\displaystyle\sum_{l\in\N}f(l)^2}\displaystyle\sum_{n\in {L}_k}f(n)^2\leq \exp\Bigg(\bigg(\dfrac{d}{8}(b-1) -\dfrac{d}{12}\log b + o(1)\bigg)\frac{(\log N)^{2-2\sigma}}{ k^2(\log_3N)^{2-2\sigma}}\Bigg).
\end{align*}
Therefore, choosing $b$ close to $1$ we obtain $3(b-1) -2\log b<0$ and summing over $k$ we obtain \eqref{30_4_5:11pm}. The proof of \eqref{30_4_5:12pm} is similar. For each $k\in\big\{1,\cdot\cdot\cdot,\big[(\log_2N)^{1/8}\big]\big\}$ and for any $b>1$ we get  
\begin{align} \label{1_5_3:36pm2}
\dfrac{1}{\displaystyle\sum_{l\in\N}f(l)^2}\displaystyle\sum_{n\in {M}_k}f(n)^2 \leq b^{-\alpha_k}\exp\bigg((b-1)\displaystyle\sum_{p\in P_{k}}f(p)^2\bigg).
\end{align}
Using the right-hand side inequality of \eqref{10_5_1:20am} we have 
\begin{align*} 
\displaystyle\sum_{p\in P_{k}}f(p)^2& \leq \bigg({\dfrac{(\log N)^{2-2\sigma}}{k^2(\log_3N)^{2-2\sigma}}}\bigg)(2+o(1)).
\end{align*}
This implies in \eqref{1_5_3:36pm2} that
\begin{align*}
\dfrac{1}{\displaystyle\sum_{l\in\N}f(l)^2}\displaystyle\sum_{n\in {L}_k}f(n)^2\leq \exp\Bigg(\big(2(b-1)-3\log b + o(1)\big)\frac{(\log N)^{2-2\sigma}}{ k^2(\log_3N)^{2-2\sigma}}\Bigg).
\end{align*}
Finally, choosing $b$ close to $1$ we obtain $2(b-1) -3\log b<0$ and summing over $k$ we obtain \eqref{30_4_5:12pm}.
\end{proof}

\smallskip

\subsection{Construction of the resonator}
Let $0\leq \beta<1$ be a fixed number and consider the positive real number $\kappa=(1-\beta)/2$. Note that $\kappa+\beta<1$. Let $\sigma$ be a positive real number and $T$ sufficiently large such that
$$
\dfrac{1}{2}\leq \sigma\leq \dfrac{1}{2}+\dfrac{1}{\log\log T}.
$$
Then we write $N=[T^{\kappa}]$. Note that $\sigma$ and $N$ satisfy the relation \eqref{1_5_9:20pm}. Now, let $\mathcal{J}$ be the set of integers $j$ such that
$$
\Big[\big(1+T^{-1}\big)^{j},\big(1+T^{-1}\big)^{j+1}\Big)\bigcap \mathcal{M} \neq \emptyset,
$$
and we define $m_j$ to be the minimum of $\big[(1+T^{-1})^{j},(1+T^{-1})^{j+1}\big)\cap \mathcal{M}$ for $j$ in $\mathcal{J}$. Consider the set
$$
\mathcal{M}':=\{m_j:j\in\mathcal{J}\}
$$
and finally we define
$$
r(m_j):=\Bigg(\displaystyle\sum_{n\in\mathcal{M},(1+T^{-1})^{j-1}\leq n \leq (1+T^{-1})^{j+2}}f(n)^2\Bigg)^{\frac{1}{2}},
$$
for every $m_j\in\mathcal{M}'$. This defines the resonator \eqref{7:40_16_4}.

\smallskip

\begin{proposition} \label{19_4_6:04pm}
We have the following properties about the resonator:
\begin{enumerate}[(i)]
\item $|\mathcal{M'}|\leq |\mathcal{M}| \leq N$.
\item $\displaystyle\sum_{m\in\mathcal{M}'}r(m)^2\leq 4\displaystyle\sum_{l\in\mathcal{M}}f(l)^2$.
\item $|R(t)|^2 \leq R(0)^2\ll T^{\kappa}\displaystyle\sum_{l\in\mathcal{M}}f(l)^2$.
\end{enumerate}
\end{proposition}
\begin{proof}
$(i)$ and $(ii)$ follow by the definition of $\mathcal{M}$, $\mathcal{M}'$ and Lemma \ref{30_7_12:26pm}. The left-hand side inequality of $(iii)$ is obvious. The right-hand side inequality of $(iii)$ follows by $(i)$, $(ii)$ and the Cauchy-Schwarz inequality.
\end{proof}

\subsection{Estimates with the resonator} The proofs of the following results are similar to the case $\sigma=\hh$. According the notation in \cite{BS} we write $\Phi(t)=e^{-t^2/2}$. Then $\widehat{\Phi}(t)=\sqrt{2\pi}\,\Phi(2\pi t)$.

\begin{lemma} \label{20_4_3:05am}
We have
$$
\int_{-\infty}^{\infty}|R(t)|^{2}\,\Phi\bigg(\dfrac{t}{T}\bigg)\,\dt \ll T\displaystyle\sum_{l\in\mathcal{M}}f(l)^2.
$$
\end{lemma}
\begin{proof} The proof is similar to \cite[Lemma 5]{BS} and we omit the details.
\end{proof}

\smallskip

\begin{lemma} \label{20_4_6:13pm}
There exists a positive constant $c>0$ such that if
$$
G(t):=\displaystyle\sum_{m=2}^{\infty}\dfrac{\Lambda(m)\,a_m}{m^{\sigma+it}\log m}
$$
is absolutely convergent and $a_m\geq 0$ for every $m\geq 2$, then
$$
\int_{-\infty}^{\infty}G(t)|R(t)|^{2}\,\Phi\bigg(\dfrac{t}{T}\bigg)\,\dt \geq c \,T\,\dfrac{(\log T)^{1-\sigma}(\log_3T)^{\sigma}}{ (\log_2T)^{\sigma}}\bigg(\min_{p\in P}a_p\bigg)\displaystyle\sum_{l\in\mathcal{M}}f(l)^2.
$$
\end{lemma}
\begin{proof} The proof follows the same outline of \cite[Lemma 7]{BS}, replacing \cite[Lemma 4]{BS} by Lemma \ref{10_5_1:57am}. We omit the details.
\end{proof}

\medskip

\section{Proof of Theorem \ref{23_4_8:21pm}}
Assume the Riemann hypothesis. We consider the parameters defined in subsection $3.1$. 

\subsection{The case $n\equiv1 \mod 2$}
We consider the entire function 
$$
K_n(z)=(-1)^{\frac{n-1}{2}}\log_2T\,\Phi(2\pi\log_2T\,z)
$$
which has Fourier transform
\begin{align}  \label{20_4_5:37pm}
\widehat{K_n}(\xi)=\dfrac{(-1)^{\frac{n-1}{2}}}{\sqrt{2\pi}}\Phi\bigg(\dfrac{\xi}{\log_2T}\bigg) \ll 1.
\end{align}
Firstly we need to estimate the following integral 
\begin{align} \label{23_4_8:30pm}
\int_{-\infty}^{\infty}\bigg(\int_{-\infty}^{\infty} S_n(\sigma,t+u)\,K_n(u)\du\bigg) |R(t)|^2\Phi\bigg(\dfrac{t}{T}\bigg)\dt.
\end{align}
This follows by the same computations as in \cite[Section 5]{BS}. We will divide \eqref{23_4_8:30pm} into $3$ integrals.
 
\smallskip

 \noindent 1. {\it First integral}: Using \eqref{8_5_8:45pm}, \eqref{15_5_4:28pm} and Fubini's theorem we get
\begin{align*} 
\int_{-T^{\beta}}^{T^{\beta}}&\int_{-\infty}^{\infty}|S_n(\sigma,t+u)\,K_n(u)|\du\,\dt  \\
&  = \int_{-T^{\beta}}^{T^{\beta}}\int_{|u|\leq T^{\beta}}|S_n(\sigma,t+u)\,K_n(u)|\du\,\dt + \int_{-T^{\beta}}^{T^{\beta}}\int_{|u|>T^\beta}|S_n(\sigma,t+u)\,K_n(u)|\du\,\dt  \\
& \ll _n\int_{-T^{\beta}}^{T^{\beta}}\int_{-2T^{\beta}}^{{2T^{\beta}}}|S_n(\sigma,u)\,K_n(u-t)|\du\,\dt + \int_{-T^{\beta}}^{T^{\beta}}\int_{|u|>T^\beta}\log(2|u|+2)|K_n(u)|\du\,\dt  \\
& \ll _n\int_{-2T^{\beta}}^{2T^{\beta}}|S_n(\sigma,u)|\,\du + T^{\beta} \ll_n T^{\beta}\log T.
\end{align*}
Hence, by Proposition \ref{19_4_6:04pm} we obtain
\begin{align} \label{19_4_6:13pm}
\int_{-T^{\beta}}^{T^{\beta}}\bigg(\int_{-\infty}^{\infty}|S_n(\sigma,t+u)\,K_n(u)|\du\bigg)\,|R(t)|^2\Phi\bigg(\dfrac{t}{T}\bigg)\dt  \ll_n T^{\beta}\log T \,R(0)^2 \ll_n T^{\beta+\kappa}\log T\displaystyle\sum_{l\in\mathcal{M}}f(l)^2.
\end{align}

 \noindent 2. {\it Second integral}: Using the fast decay of $\Phi(t)$, \eqref{8_5_8:45pm} and Proposition \ref{19_4_6:04pm}, it follows that
\begin{align} \label{19_4_6:17pm}
\begin{split}
\int_{|t|>T\log T} &\bigg(\int_{-\infty}^{\infty}|S_n(\sigma,t+u)\,K_n(u)|\du\bigg)\,|R(t)|^2\Phi\bigg(\dfrac{t}{T}\bigg)\dt   \\ 
&\ll T^{\kappa}e^{-\frac{(\log T)^2}{4}}\Bigg(\int_{|t|>T\log T}\int_{-\infty}^{\infty}|S_n(\sigma,t+u)\,K_n(u)|\du\,\Phi\bigg(\dfrac{t}{2T}\bigg)\dt \Bigg)\,\displaystyle\sum_{l\in\mathcal{M}}f(l)^2 \\
& = o(1)\displaystyle\sum_{l\in\mathcal{M}}f(l)^2.
\end{split}
\end{align}

\noindent 3. {\it Third integral}: 
\begin{align} \label{20_4_2:53am}
\begin{split}
& \int_{T^{\beta}\leq |t|\leq T\log T}\bigg(\int_{-\infty}^{\infty} S_n(\sigma,t+u)\,K_n(u)\du\bigg) \,|R(t)|^2\Phi\bigg(\dfrac{t}{T}\bigg)\dt \\
& \, \, \, \, \, \, \, \, \, \, \, \, \, \,  = \int_{T^{\beta}\leq |t|\leq T\log T}\bigg(\int_{\frac{T^{\beta}}{2}\leq |t+u|\leq 2T\log T} S_n(\sigma,t+u)\,K_n(u)\du\bigg)\,|R(t)|^2\Phi\bigg(\dfrac{t}{T}\bigg)\dt   \\
& \, \, \, \, \, \, \, \, \, \, \, \, \, \, \, \, \, \, \, + \int_{T^{\beta}\leq |t|\leq T\log T}\bigg(\int_{\{|u+t|<\frac{T^\beta}{2}\}\cup\{|u+t|>2T\log T\}} S_n(\sigma,t+u)\,K_n(u)\du\bigg)\,|R(t)|^2\Phi\bigg(\dfrac{t}{T}\bigg)\dt.
\end{split}
\end{align}
Now using \eqref{8_5_8:45pm} and Lemma \ref{20_4_3:05am},  the last integral can be bounded by
\begin{align} \label{20_4_3:13am}
\begin{split}
\int_{T^{\beta}\leq |t|\leq T\log T} & \int_{\{|u+t|<\frac{T^\beta}{2}\}\cup\{|u+t|>2T\log T\}} |S_n(\sigma,t+u)\,K_n(u)|\du \,|R(t)|^2\Phi\bigg(\dfrac{t}{T}\bigg)\dt  \\
& \ll \int_{T^{\beta}\leq |t|\leq T\log T}\int_{\{|u|<\frac{T^\beta}{2}\}\cup\{|u|>2T\log T\}}|S_n(\sigma,u)\,K_n(u-t)|\du \,|R(t)|^2\Phi\bigg(\dfrac{t}{T}\bigg)\dt \\
& \leq \int_{T^{\beta}\leq |t|\leq T\log T}\int_{\{|u|<\frac{T^\beta}{2}\}\cup\{|u|>2T\log T\}} \Big|S_n(\sigma,u)\,K_n\Big(\frac{u}{2}\Big)\Big|\du \,|R(t)|^2\Phi\bigg(\dfrac{t}{T}\bigg)\dt   \\
& \ll_n \int_{T^{\beta}\leq |t|\leq T\log T}|R(t)|^2\Phi\bigg(\dfrac{t}{T}\bigg)\dt \ll T\displaystyle\sum_{l\in\mathcal{M}}f(l)^2.
\end{split}
\end{align}
Inserting \eqref{20_4_3:13am} in \eqref{20_4_2:53am} we obtain that 
\begin{align} \label{20_4_3:19am}
\begin{split}
&\int_{T^{\beta}\leq |t|\leq T\log T}\bigg(\int_{-\infty}^{\infty}S_n(\sigma,t+u)\,K_n(u)\du\bigg)\,|R(t)|^2\Phi\bigg(\dfrac{t}{T}\bigg)\dt \\
& \, \, \, \, \, \, = \int_{T^{\beta}\leq |t|\leq T\log T}\Bigg(\int_{\frac{T^{\beta}}{2}\leq |t+u|\leq 2T\log T} S_n(\sigma,t+u)\,K_n(u)\du\Bigg) \,|R(t)|^2\Phi\bigg(\dfrac{t}{T}\bigg)\dt +  O_n(T)\displaystyle\sum_{l\in\mathcal{M}}f(l)^2.
\end{split}
\end{align}
Therefore, combining \eqref{19_4_6:13pm}, \eqref{19_4_6:17pm} and \eqref{20_4_3:19am} we have that the integral in  \eqref{23_4_8:30pm} can be written as
\begin{align}  \label{15_5_8:00pm}
\begin{split}
\int_{-\infty}^{\infty} \bigg(\int_{-\infty}^{\infty} S_n(\sigma,&\, t+u) \,K_n(u)\du \bigg)\,|R(t)|^2\Phi\bigg(\dfrac{t}{T}\bigg)\dt +  O_n(T)\displaystyle\sum_{l\in\mathcal{M}}f(l)^2 \\
& = \int_{T^{\beta}\leq |t|\leq T\log T}\Bigg(\int_{\frac{T^{\beta}}{2}\leq |t+u|\leq 2T\log T} S_n(\sigma,t+u)\,K_n(u)\du\Bigg) \,|R(t)|^2\Phi\bigg(\dfrac{t}{T}\bigg)\dt.
  \end{split}
\end{align}

\smallskip

Now we consider two subcases: 
\subsubsection{The subcase n\,$\equiv 1\mod 4$}. In this case note that $K_n(u)\geq 0$ for all $u\in\R$. Then by Lemma \ref{20_4_3:05am} and the fact that $S_n(\sigma, t)$ is an even function we obtain in \eqref{15_5_8:00pm} 
\begin{align} \label{20_4_3:43am}
\begin{split}
\int_{-\infty}^{\infty} \bigg(\int_{-\infty}^{\infty} S_n(\sigma, t+u)\,K_n(u)\du\bigg)& \,|R(t)|^2\Phi\bigg(\dfrac{t}{T}\bigg)\dt +  O_n(T)\displaystyle\sum_{l\in\mathcal{M}}f(l)^2 \\
 & \, \, \, \, \, \,  \, \, \, \, \, \, \, \, \, \leq b\,T\Bigg(\max_{\frac{T^{\beta}}{2}\leq t\leq 2T\log T}S_n(\sigma,t)\Bigg)\displaystyle\sum_{l\in\mathcal{M}}f(l)^2,
\end{split}
\end{align}
for some constant $b>0$. We define
\begin{align} \label{23_4_8:31pm}
G_{n}(t)=\displaystyle\sum_{m=2}^{\infty}\dfrac{\Lambda(m)}{\pi \,m^{\sigma+it}(\log m)^{n+1}}\widehat{K_n}\bigg(\dfrac{\log m}{2\pi}\bigg).
\end{align}
By Proposition \ref{20_4_3:44am} and \eqref{20_4_5:37pm} observe that
$$
\int_{-\infty}^{\infty}S_{n}(\sigma,t+u)\,K_n(u)\du= \re G_{n}(t)+O_n\big(V_{\frac{1}{2}}(t)+1\big),
$$
for $t\neq 0$. Therefore, the integral on the left-hand side of  \eqref{20_4_3:43am} takes the form
\begin{align} \label{20_4_5:59pm}
\begin{split}
\int_{-\infty}^{\infty}\bigg(\int_{-\infty}^{\infty} & S_n(\sigma,t+u)\,K_n(u)\du\bigg)\,|R(t)|^2\Phi\bigg(\dfrac{t}{T}\bigg)\dt  \\
 & = \re\int_{-\infty}^{\infty} G_{n}(t)|R(t)|^2\Phi\bigg(\dfrac{t}{T}\bigg)\dt + O_n\bigg(\int_{-\infty}^{\infty} \big(V_{\frac{1}{2}}(t)+1\big)|R(t)|^2\Phi\bigg(\dfrac{t}{T}\bigg)\dt\bigg).
\end{split}
\end{align}
Using Proposition \ref{19_4_6:04pm}, Lemma \ref{20_4_3:05am} and the definition of $V_{\frac{1}{2}}(t)$ we get
\begin{align} \label{23_4_8:33pm}
\int_{-\infty}^{\infty}\big(V_{\frac{1}{2}}(t)+1\big)|R(t)|^2\Phi\bigg(\dfrac{t}{T}\bigg)\dt \ll T\displaystyle\sum_{l\in\mathcal{M}}f(l)^2.
\end{align}
Therefore using \eqref{20_4_5:59pm} and \eqref{23_4_8:33pm} we have
\begin{align} \label{20_4_7:11pm}
b\,T\Bigg(\max_{\frac{T^{\beta}}{2}\leq t\leq 2T\log T}S_n(\sigma,t)\Bigg)\displaystyle\sum_{l\in\mathcal{M}}f(l)^2\geq \re\int_{-\infty}^{\infty} G_{n}(t)|R(t)|^2\Phi\bigg(\dfrac{t}{T}\bigg)\dt +  O_n(T)\displaystyle\sum_{l\in\mathcal{M}}f(l)^2.
\end{align}
Now using Lemma \ref{20_4_6:13pm} (note that $\widehat{K_n}(t)$ is a positive real function) with 
$$a_m=\widehat{K_n}\bigg(\frac{\log m}{2\pi}\bigg)\frac{1}{\pi(\log m)^n},$$ for all $m\geq 2$ we obtain that
\begin{align} \label{20_4_7:10pm33}
\re\int_{-\infty}^{\infty} G_n(t)|R(t)|^2\Phi\bigg(\dfrac{t}{T}\bigg)\dt & \geq  c \,T\dfrac{(\log T)^{1-\sigma}(\log_3T)^{\sigma}}{ (\log_2T)^{\sigma}}\bigg(\min_{p\in P}\widehat{K_n}\bigg(\frac{\log p}{2\pi}\bigg)\frac{1}{(\log p)^n}\bigg)\displaystyle\sum_{l\in\mathcal{M}}f(l)^2,
\end{align}
for some constant $c>0$. Note that \eqref{20_4_6:46pm} and \eqref{20_4_5:37pm} imply \begin{align*} 
\min_{e\log N\log_{2}N < p \leq \exp\big((\log_2N)^{1/8}\big)\log N\log_2N}\widehat{K_n}\bigg(\frac{\log p}{2\pi}\bigg)\frac{1}{(\log p)^n} \gg \dfrac{1}{(\log_2 T)^n}.
\end{align*}
Inserting this in \eqref{20_4_7:10pm33}, we obtain in \eqref{20_4_7:11pm} that (after simplification) 
$$
\max_{\frac{T^{\beta}}{2}\leq t\leq 2T\log T}S_n(\sigma,t) \geq c_n\,\dfrac{(\log T)^{1-\sigma}(\log_3T)^{\sigma}}{ (\log_2T)^{\sigma+n}} + O_n(1),
$$
for some constant $c_n>0$. After a trivial adjustment, changing $T$ to $T/2\log T$ and making $\beta$ slightly smaller, we obtain the restriction $T^\beta\leq t \leq T$.

\subsubsection{The subcase n\,$\equiv 3\mod 4$} In this case note that $K_n(u)\leq 0$ for all $u\in\R$. Similar to \eqref{20_4_3:43am}, using the fact that $S_n(t)$ is an even function we find that 
\begin{align} \label{20_4_3:43am55}
\begin{split}
\int_{-\infty}^{\infty} \bigg(\int_{-\infty}^{\infty} S_n(\sigma,t+u)\,K_n(u)\du\bigg)& \,|R(t)|^2\Phi\bigg(\dfrac{t}{T}\bigg)\dt +  O_n(T)\displaystyle\sum_{l\in\mathcal{M}}f(l)^2 \\
 & \, \, \, \, \, \,  \, \, \, \, \, \, \, \, \, \leq b\,T\Bigg(\max_{\frac{T^{\beta}}{2}\leq t\leq 2T\log T}|S_n(\sigma,t)|\Bigg)\displaystyle\sum_{l\in\mathcal{M}}f(l)^2,
\end{split}
\end{align}
for some constant $b>0$. Using the function $G_n$ defined in \eqref{23_4_8:31pm}, by Proposition \ref{20_4_3:44am} and \eqref{20_4_5:37pm} we get
$$
\int_{-\infty}^{\infty}S_{n}(\sigma,t+u)\,K_n(u)\du= -\re G_{n}(t)+O_n\big(V_{\frac{1}{2}}(t)+1\big).
$$
A similar analysis as in the previous case shows that, by Lemma \ref{20_4_6:13pm} (note that $-\widehat{K_n}(t)$ is a positive real function)
\begin{align} \label{20_4_7:10pm}
\re\int_{-\infty}^{\infty} -G_{n}(t)|R(t)|^2\Phi\bigg(\dfrac{t}{T}\bigg)\dt & \geq  c \,T\dfrac{(\log T)^{1-\sigma}(\log_3T)^{\sigma}}{ (\log_2T)^{\sigma}}\bigg(\min_{p\in P}-\widehat{K_n}\bigg(\frac{\log p}{2\pi}\bigg)\frac{1}{(\log p)^n}\bigg)\displaystyle\sum_{l\in\mathcal{M}}f(l)^2,
\end{align}
for some constant $c>0$. By \eqref{20_4_6:46pm} and \eqref{20_4_5:37pm} we have
 \begin{align*} 
\min_{e\log N\log_{2}N < p \leq \exp\big((\log_2N)^{1/8}\big)\log N\log_2N}-\widehat{K_n}\bigg(\frac{\log p}{2\pi}\bigg)\frac{1}{(\log p)^n} \gg \dfrac{1}{(\log_2 T)^n}.
\end{align*}
Inserting this in \eqref{20_4_7:10pm} we obtain in \eqref{20_4_3:43am55} that (after simplification) 
$$
\max_{\frac{T^{\beta}}{2}\leq t\leq 2T\log T}|S_n(\sigma,t)| \geq c_n\,\dfrac{(\log T)^{1-\sigma}(\log_3T)^{\sigma}}{ (\log_2T)^{\sigma+n}} + O_n(1),
$$
for some constant $c_n>0$. After the same trivial adjustment of $T$ and $\beta$ as in the preceding case we obtain the desired result.

\subsection{The case $n\equiv 0 \mod 2$}
We consider the entire function 
$$
K_n(z)=(-1)^{\frac{n}{2}+1}(\log_2T)^2\,z\,\Phi(2\pi\log_2T\,z)
$$
which has Fourier transform
\begin{align}  \label{20_4_5:37pm2}
\widehat{K_n}(\xi)=\dfrac{(-1)^{\frac{n}{2}}\,i}{(2\pi)^{\frac{3}{2}}(\log_2 T)}\xi\,\Phi\bigg(\dfrac{\xi}{\log_2T}\bigg) \ll 1.
\end{align}
The analysis in this case is similar to the case $n\equiv 3 \mod4$. Using the fact that $S_n(t)$ is an odd function we obtain that \eqref{20_4_3:43am55} holds. Using the function $G_{n}$ defined in \eqref{23_4_8:31pm}, by Proposition \ref{20_4_3:44am} and \eqref{20_4_5:37pm2} note that
$$
\int_{-\infty}^{\infty}S_{n}(\sigma,t+u)\,K_n(u)\du=(-1)^{\frac{n}{2}}\im G_{n}(t)+O_n\big(V_{\frac{1}{2}}(t)+1\big).
$$
This implies that in \eqref{20_4_3:43am55} we obtain 
\begin{align*} 
b\,T\Bigg(\max_{\frac{T^{\beta}}{2}\leq t\leq 2T\log T}\big|S_n(\sigma,t)\big|\Bigg)\displaystyle\sum_{l\in\mathcal{M}}f(l)^2\geq \re \int_{-\infty}^{\infty} (-1)^{\frac{n}{2}+1}\,i\,G_{n}(t)|R(t)|^2\Phi\bigg(\dfrac{t}{T}\bigg)\dt +  O_n(T)\displaystyle\sum_{l\in\mathcal{M}}f(l)^2  ,
\end{align*}
for some constant $b>0$. Now, using Lemma \ref{20_4_6:13pm} (note that $i(-1)^{\frac{n}{2}+1}\widehat{K_n}(t)$ is a positive real function for $t\geq 0$) it follows that
\begin{align} \label{20_4_7:10pm233}
\begin{split}
 T\Bigg(\max_{\frac{T^{\beta}}{2}\leq t\leq 2T\log T}&\big|S_n(\sigma,t)\big|\Bigg)\displaystyle\sum_{l\in\mathcal{M}}f(l)^2  \\
& \geq c \,T\dfrac{(\log T)^{1-\sigma}(\log_3T)^{\sigma}}{ (\log_2T)^{\sigma}}\bigg(\min_{p\in P}\im\bigg\{(-1)^{\frac{n}{2}}\widehat{K_n}\bigg(\frac{\log p}{2\pi}\bigg)\frac{1}{(\log p)^n}\bigg)\bigg\}\displaystyle\sum_{l\in\mathcal{M}}f(l)^2,
\end{split}
\end{align}
for some constant $c>0$. By \eqref{20_4_6:46pm} and \eqref{20_4_5:37pm2} we have
 \begin{align*} 
\min_{e\log N\log_{2}N < p \leq \exp\big((\log_2N)^{1/8}\big)\log N\log_2N}\im\bigg\{(-1)^{\frac{n}{2}}\widehat{K_n}\bigg(\frac{\log p}{2\pi}\bigg)\frac{1}{(\log p)^n}\bigg\} \gg \dfrac{1}{(\log_2 T)^n}.
\end{align*}
Inserting this in \eqref{20_4_7:10pm233} and doing the same procedure as in the previous cases we obtain the desired result.

\medskip

\section*{Acknowledgements}
I would like to thank Emanuel Carneiro for all the motivation and insightful conversations
on this subject. The author also acknowledges support
from FAPERJ-Brazil.

\end{document}